\documentclass[12pt]{amsart}

\usepackage{mathrsfs}
\usepackage{graphicx}

\usepackage{amssymb}

\usepackage{amsthm}
\usepackage{longtable}
\usepackage{multirow}
\usepackage{makecell}
\usepackage[all,cmtip]{xy}
\usepackage{geometry}

\usepackage{tikz}
\usepackage{nicematrix}
\usepackage{amsmath}

\NiceMatrixOptions{code-for-first-row = \scriptstyle,code-for-first-col = \scriptstyle }
\newtheorem{thm}{Theorem}[section]
\newtheorem{cor}[thm]{Corollary}
\newtheorem{lem}[thm]{Lemma}
\newtheorem{prop}[thm]{Proposition}

\theoremstyle{definition}
\newtheorem{defn}[thm]{Definition}

\newtheorem{ques}[thm]{Question}
\theoremstyle{remark}
\newtheorem{rem}[thm]{Remark}
\numberwithin{equation}{section}

\newcommand{\Z}{\mathbb Z}

\makeatletter
\@namedef{subjclassname@2020}{\textup{2020} Mathematics Subject Classification}
\makeatother

\begin{document}

\title{Deciding whether a mapping torus is of full rank}

\author{Juemin Lin}
\address{Department of Mathematics, Soochow University, Suzhou 215006, CHINA}
\email{20224207039@stu.suda.edu.cn}

\author{Jianchun Wu$^*$}
\thanks{* The second author is the corresponding author.}
\address{Department of Mathematics, Soochow University, Suzhou 215006, CHINA}
\email{wujianchun@suda.edu.cn}


\subjclass[2020]{20F10, 20G30, 20E22}

\keywords{mapping torus, generating orbit sets, rank, generalized linear group}

\begin{abstract}
The mapping torus induced by an automorphism $\phi$ of the free abelian group $\Z^n$ is a semi-direct product $G=\Z^n\rtimes_\phi \Z$. We show that whether the rank of $G$ is equal to $n+1$ is decidable. As a corollary, the rank of $\Z^3\rtimes_\phi \Z$ is decidable.
\end{abstract}

\maketitle

\section{Introduction}

Let $\phi$ be an automorphism of a group $F$, the semi-direct product $G=F\rtimes_\phi\Z$ is usually called a mapping torus group. There are very few known examples for which one knows how to compute the rank (i.e. the minimum cardinality of a generating set) of this kind of groups.

J. Souto \cite{So} showed that when $F$ is the fundamental group $S_g$ of a closed orientable surface of genus $g\geq2$ and $\phi$ is an automorphism of $S_g$ representing a pseudo-Anosov mapping class, then $S_g\rtimes_{\phi^k}\Z$ with $k$ large enough is of full rank (i.e. the rank is $2g+1$). Biringer and Souto \cite{BS} proved that if $\phi$ satisfies some geometric conditions then $S_g\rtimes_{\phi}\Z$ is of full rank. Dowdall and Taylor \cite{DT} generalized Souto's result to a large class of hyperbolic group extensions. When $F$ is the  free abelian group $\Z^n$ of rank $n$, Levitt and Metaftsis showed
\begin{thm}[\cite{LM}, Theorem 1.1]\label{rk2}
There is an algorithm that decides whether $\Z^n\rtimes_\phi \Z$ has rank 2 or not.
\end{thm}
They also proved a similar result of J. Souto that the rank of $\Z^n\rtimes_{\phi^k} \Z$ is at least 3 for large enough $k$, and the authors of \cite{AZ} gave an alternative approach generalizing it. The proof of Theorem \ref{rk2} depends on the following observation:

An element $v\in \Z^n$  can be written as a column vector $v=[v_1,...,v_n]^T$ when a basis of $\Z^n$ is fixed, and the automorphism $\phi$ induces an element $A$ in the general linear group $GL_n(\Z)$ such that $\phi(v)=Av$ for any $v\in \Z^n$. The $A$-orbit of $v$ is $\{A^k v\mid k\in \Z\}$. Denote by $m_A$ the minimum number of $A$-orbits needed to generate $\Z^n$. Levitt and Metaftsis observed that the rank of $\Z^n\rtimes_\phi \Z$ is equal to $m_A+1$ and gave a way to decide whether $m_A=1$. As a corollary, one can compute the rank of $\Z^2\times_{\phi}\Z$ or $F_2\times_{\phi}\Z$ where $F_2$ is the free group of rank 2, but the following  question is left:

\begin{ques}[Levitt-Metaftsis]
Is there an algorithm that, given an automorphism of $\phi$ of $\Z^n$, computes the rank of $G=\Z^n\rtimes_\phi \Z$?
\end{ques}

There are few relevant results about this question apart from \cite{WZ} which gave a complete list of the ranks of those groups for $n=3$ or $4$ with $\phi$ periodic. In the present paper, we prove

\begin{thm}\label{mthm}
Whether the rank of $\Z^n\rtimes_\phi \Z$ is equal to n+1  is decidable.
\end{thm}

It is  trivial for $n=1$ while \cite{LM} has provided an algorithm for $n=2$. One can see Theorem \ref{mthm2} which gives a simple way to decide it for $n\geq 3$. Since the rank of $\Z^n\rtimes_\phi \Z$ is not less than 2, by Theorem \ref{rk2} and \ref{mthm} we have the following
\begin{cor}
The rank of $\Z^3\rtimes_\phi \Z$ is decidable.
\end{cor}


\section{Preliminaries}

\begin{defn}
Let $A$ be an $n\times n$ integer matrix and $S\subseteq\Z^n$. The orbit subgroup $OG_A(S)$ on $S$ by $A$ is the subgroup of $\Z^n$ generated by $\{A^k v\ |\ k\geq 0, v\in S\}$. If $OG_A(S)$ is the full group $\Z^n$, we call $S$ a generating orbit set of $A$. Among all generating orbit sets of $A$ those having minimal cardinalities are called minimal generating orbit sets. We denote the cardinality of a minimal generating orbit set by $m_A$ , that is $m_A=\min\{\# S\mid OG_A(S)=\Z^n\}$.
\end{defn}

\begin{lem}\label{cor:l}
$m_A=m_{A-\lambda I}$ for any $\lambda \in \Z$.
\end{lem}
\begin{proof}
Let $S$ be a minimal generating orbit set of $A-\lambda I$, since $(A-\lambda I)^k v$ is an integral linear combination of $\{v, Av,\cdots, A^kv\}$ for $k\geq 0$ and $v\in S$, we have $\Z^n=OG_{A-\lambda I}(S)\subseteq OG_A(S)$ which means $S$ is a generating orbit set of $A$, hence $m_{A-\lambda I}\geq m_A$. Similarly, $m_{A}= m_{(A-\lambda I)+\lambda I}\geq m_{A-\lambda I}$.
\end{proof}

For a matrix $A\in GL_n(\Z)$, by Cayley-Hamilton Theorem, $A^k v$ is an integral linear combination of $v, Av,\cdots, A^{n-1}v$ for any $v\in \Z^n$ and $k\in \Z$, hence the $A$-orbit $\{A^kv\mid k\in \Z\}$ of $v$ is a subset of $OG_A(v)$ and $m_A$ is the minimum number of $A$-orbits needed to generate $\Z^n$. The following lemma is proved by Levitt and Metaftsis.
\begin{lem}[\cite{LM}, Corollary 2.4]\label{LM}
Let $\phi$ be an automorphism of $\Z^n$, then the rank of $\Z^n\rtimes_{\phi}\Z$ is equal to $m_A+1$ where $A\in GL_n(Z)$ is the matrix induced by $\phi$ when a basis of $\Z^n$ is fixed.
\end{lem}





Suppose $A, B$ are two $n\times n$ integer matrices, if there exists $X\in GL_n(\Z)$ such that $B=XAX^{-1}$, then we say $A$ is integrally conjugate to $B$. Throughout this paper, conjugation always means integral conjugation. The following lemma shows that $m_A$ is a conjugation invariant.

\begin{lem}\label{lem:co}
Suppose  $A, B$ are two $n\times n$ integer matrices such that $A$ is integrally conjugate to $B$, then $m_A=m_B$.
\end{lem}
\begin{proof}
If $S$ is a minimal generating orbit set of $A$, since $B=XAX^{-1}$ for some $X\in GL_n(\mathbb{Z})$, it is obvious that $\{Xv\mid v\in S\}$ is a generating orbit set of $B$, so $m_A\ge m_{B}$. Similarly, $m_{B}\ge m_{X^{-1}BX}=m_A$.
\end{proof}

Let $C$ be a finite set of integers, the greatest common divisor of absolute values of all elements in $C$ is denoted by $\gcd(C)$. We assume any prime number divides 0 and $\gcd(0)=0$.
\begin{lem}\label{lem:n}
Let $A$ be an $n\times n$ integer matrix such that $\gcd(A)\neq 1$, then $m_A=n$.
\end{lem}
\begin{proof}
Let $d=\gcd(A)$. If $d=0$, then $A=0$, hence $m_A=n$. If $d\neq 0$, let $S=\{v_1,\cdots,v_m\}$ be a minimal generating orbit set of $A$ (note that $m\leq n$), then $\mathbb{Z}^n$ is spanned by $\{A^k v_j\mid k\geq 0, j=1,\cdots, m\}$. Denote by $\phi: \Z^n\to \Z_d^n$ the canonical module $d$ homomorphism, then $\Z_d^n$ is spanned by $\{\phi(v_j)\mid j=1,\cdots, m\}$ since $\phi(A^k v_j)=0$ when $k>0$. Thus $m\geq n$ and we have $m_A=n$.
\end{proof}

The following lemma will be frequently used in this paper.




\begin{lem}\label{lem:g}
Let $X$ be an $n\times m$ integer matrix, then $\gcd(X)=\gcd(AXB)$ for any $A\in GL_n(\mathbb{Z})$ and $B\in GL_m(\mathbb{Z})$.
\end{lem}
\begin{proof}
It is trivial when $\gcd(X)=0$. Since all entries in $AXB$ are integral linear combinations of entries in $X$, then $\gcd(X)$ divides $\gcd(AXB)$. Similarly, $\gcd(AXB)$ divides $\gcd(A^{-1}(AXB)B^{-1})=\gcd(X)$. So we have $\gcd(X)=\gcd(AXB)$.
\end{proof}

\section{Proof of Theorem \ref{mthm}}\label{fact}

In this section, we will prove $m_A=n$ if and only if $\gcd(A-a_{11}I)\neq 1$ for $A=(a_{ij})\in GL_n(\Z)$ with $n\geq 3$. By Lemma \ref{LM}, it provides a way to  decide whether the rank of a mapping torus $\Z^n\times_{\phi} Z$ is $n+1$.
%

%

\begin{defn}
We call an $n\times n$ integer matrix $H=(h_{ij})$ with $n\geq 3$ a type $\mathcal{H}$ matrix if $h_{i1}=0$ for $i=3,...,n$. That is to say the shape of  $H$ is
    $$ \begin{bmatrix}
        h_{11}& h_{12}& \cdots   & h_{1n} \\
        h_{21}& h_{22}& \cdots  & h_{2n} \\
        0 & h_{32} & \cdots & h_{3n} \\
        \vdots & \vdots & \vdots & \vdots \\
        0 & h_{n2} & \cdots & h_{nn}
    \end{bmatrix}.$$
Moreover, if $h_{11}=0$, then we say $H$ is a type $\mathcal{H}_0$ matrix. A matrix $H=(h_{ij})$ is called  type $\mathcal{H}_n$ if $H$ is of type $\mathcal{H}_0$ and $\gcd(h_{21}, h_{1n},\cdots,h_{nn})=1$.
\end{defn}

\begin{prop}\label{prop:h}
Let $A=(a_{ij})$ be an $n\times n$ integer matrix with $n\geq 3$, then $A$ is integrally conjugate to a matrix $B=(b_{ij})$ of type $\mathcal{H}$ with $b_{11}=a_{11}$.
\end{prop}
\begin{proof}
If $A$ is not of type $\mathcal{H}$, then there exists some $3\leq k\leq n$ such that $a_{k1}\neq 0$. Without loss of generality, we can assume $a_{21}\neq 0$, otherwise let

\[X=\begin{pNiceMatrix}[last-row,last-col,nullify-dots,xdots/line-style={dashed,blue}]
1& \Vdots & & & & \Vdots  \\
 \Cdots&    0 & \Cdots  & & &   1 & \Cdots& &   &   \leftarrow 2 \\
&  & 1 \\
&  \Vdots & & \Ddots[line-style=standard] & &   \\
& \Vdots & & & 1&\Vdots \\
 \Cdots &    -1  & \Cdots  & &   &   0 &\Cdots & &   &   \leftarrow k \\
 &\Vdots & & & &\Vdots & 1 \\
 & & & & & & & \Ddots[line-style=standard]  \\
 & & & & & & & & 1 \\
 &   \overset{\uparrow}{2} & & & &   \overset{\uparrow}{k} \\
\end{pNiceMatrix},\]
then $X\in GL_n(\Z)$, the $(2,1)$ entry of $XAX^{-1}$ is not 0 and the $(1,1)$ entry is $a_{11}$.

Now there exist two integers $s, t$ such that $sa_{21}+ta_{k1}=d$ where $d=\gcd(a_{21},a_{k1})$.
Let


\[Y=\begin{pNiceMatrix}[last-row,last-col,nullify-dots,xdots/line-style={dashed,blue}]
1& \Vdots & & & & \Vdots  \\
 \Cdots&    s & \Cdots  & & &   t & \Cdots& &   &   \leftarrow 2 \\
&  & 1 \\
&  \Vdots & & \Ddots[line-style=standard] & &   \\
& \Vdots & & & 1&\Vdots \\
 \Cdots &    \displaystyle-\frac{a_{k1}}{d}  & \Cdots  & &   &   \displaystyle\frac{a_{21}}{d} &\Cdots & &   &   \leftarrow k \\
 &\Vdots & & & &\Vdots & 1 \\
 & & & & & & & \Ddots[line-style=standard]  \\
 & & & & & & & & 1 \\
 &   \overset{\uparrow}{2} & & & &   \overset{\uparrow}{k} \\
\end{pNiceMatrix},\]
one can verify $Y$ is in $GL_n(\Z)$, the $(k,1)$ entry of $\widetilde{A}=YAY^{-1}$ is 0, the $(2,1)$ entry is $d$ and  the $(i,1)$ entry is $a_{i1}$ for $i\neq 2$ and $i\neq k$.


The above procedure applies to $\widetilde{A}$  and so on. Finally we get a matrix $B=(b_{ij})$ of type $\mathcal{H}$ which is conjugate to $A$ with $b_{11}=a_{11}$.
\end{proof}
\begin{rem}
A similar method can be used to prove  each $n\times n$ integer matrix is integrally conjugate to an upper Hessenberg matrix (see \cite{Ka} Section 21.2 for another proof), but we won't use this fact.
\end{rem}

\begin{lem}\label{lem:v}
    Let $v_1$ be a non-zero element in $\mathbb{Z}^n$ with $n\geq 1$, then for any $v_2,v_3\in\Z^n$, there exists $k\in \mathbb{Z}$ such that $\gcd(v_1,v_2,v_3)=\gcd(v_1,v_3+kv_2)$.
\end{lem}
\begin{proof}
If $\gcd(v_1,v_2,v_3)=1$,  denote by $\mathcal{P}$ the set $\{p \text{ is prime}, ~p\mid \gcd(v_1)\}$ which  can be divided into three disjoint subsets as follows
    \begin{align*}
        &\mathcal{P}_1=\{p \in \mathcal{P},~p\mid \gcd(v_3)\},\\
        &\mathcal{P}_2=\{p \in \mathcal{P},~p\nmid \gcd(v_3), p\mid \gcd(v_2) \},\\
        &\mathcal{P}_3=\mathcal{P}-\mathcal{P}_1\cup\mathcal{P}_2.
    \end{align*}
Note that $v_1$ is non-zero, $\mathcal{P}_3\subset \mathcal{P}$ is finite. Let $$k=\begin{cases}
             1, & \mbox{if } \mathcal{P}_3=\varnothing \\
             \prod_{p\in \mathcal{P}_3}p, & \mbox{otherwise}.\end{cases}$$

   If $p$ is a prime number dividing $\gcd(v_1,v_3+kv_2)$, then $p\in \mathcal{P}=\mathcal{P}_1\cup\mathcal{P}_2\cup\mathcal{P}_3$. Moreover, if $p\in \mathcal{P}_2\cup\mathcal{P}_3$, then $p$ divides $\gcd(kv_2)$, so $p$ divides $\gcd(v_3)$, but $p\notin\mathcal{P}_1$, we get a contradiction.  Hence $p\in\mathcal{P}_1$, $p$ divides $\gcd(v_3)$, then $p$ divides $\gcd(kv_2)$. Note that $p$ does not divide $k$, so $p$ divides $\gcd(v_2)$, we have $p \mid \gcd(v_1,v_2,v_3)=1$, a contradiction. Thus $\gcd(v_1,v_3+kv_2)=1=\gcd(v_1,v_2,v_3)$.

    If $\gcd(v_1,v_2,v_3)=d\neq 1$ then $\gcd(v_1/d,v_2/d,v_3/d)=1$, the above argument shows that there exists $k\in \mathbb{Z}$ such that $\gcd(v_1/d,v_3/d+kv_2/d)=\gcd(v_1/d,v_2/d,v_3/d)$, hence $\gcd(v_1,v_3+kv_2)=d\gcd(v_1/d,v_3/d+kv_2/d)=d=\gcd(v_1,v_2,v_3)$.
\end{proof}

\begin{prop}\label{lem:hn}
Suppose $A$ is an $n\times n$ integer matrix of type $\mathcal{H}_0$ with $\gcd(A)=1$, then $A$ is integrally conjugate to a matrix of type $\mathcal{H}_n$.
\end{prop}
\begin{proof}
    Denote by $v_j$ the $j$-th column vector of matrix $A=(a_{ij})$. There are two cases.

Case 1: $v_1$ is zero, that is to say $\gcd(v_2,\cdots,v_n)=1$. Hence there exist $B\in GL_n(\mathbb{Z})$ and $C\in GL_{n-1}(\mathbb{Z})$ such that $B[v_2,\cdots,v_n]C$ is the Smith normal form (one can see \cite{Ne} for more details) of $[v_2,\cdots,v_n]$ whose shape is
$$\begin{bmatrix}
        1 & * &\cdots & * \\
        0& *&\cdots & *\\
        \vdots& \vdots & \ddots& \vdots\\
        0& *&\cdots & *\\
    \end{bmatrix}_{n\times (n-1)}.$$
Let $D$ be the permutation matrix $$\begin{bmatrix}
        0 &\cdots  & 0 &1 \\
        1& \cdots &0& 0\\
        \vdots& \ddots & \vdots& \vdots\\
        0&\cdots & 1 & 0\\
    \end{bmatrix}_{(n-1)\times (n-1)},$$
then $B[v_2,\cdots,v_n]CD=[\ast ,\cdots,\ast ,e_1]$ where $e_1=[1,0,\cdots,0]^T$.

Let $T=\begin{bmatrix}
    1 & 0\\
    0 & CD
\end{bmatrix}$, then
\begin{align*}
    T^{-1}AT & = T^{-1}[v_1,v_2,\cdots,v_n]T\\
    &=T^{-1}[v_1,[v_2,\cdots,v_n]CD]\\
    &=T^{-1}[v_1,B^{-1}[\ast ,\cdots,\ast ,e_1]] \\ &=[v_1,\ast ,\cdots,\ast ,T^{-1}B^{-1}e_1].
\end{align*}
By Lemma \ref{lem:g}, we have
$\gcd(T^{-1}B^{-1}e_1)=\gcd(e_1)=1$, hence $\gcd(v_1,T^{-1}B^{-1}e_1)=1$ and $T^{-1}AT$ is a matrix of type $\mathcal{H}_n$.

Case 2: $v_1$ is non-zero.

We construct $n-1$ vectors $\tilde{v}_2,\cdots,\tilde{v}_n$ inductively such that $\gcd(v_1,\cdots, v_j)=\gcd(v_1,\tilde{v}_j)$ for $j=2,\cdots,n$ as follows.

    Let $\tilde{v}_2=v_2$, then $\gcd(v_1, v_2)=\gcd(v_1,\tilde{v}_2)$.  Suppose $\gcd(v_1,\cdots, v_j)=\gcd(v_1,\tilde{v}_j)$, by Lemma \ref{lem:v}, there exists $k_{j+1}\in\mathbb{Z}$, such that $$\gcd(v_1,\tilde{v}_j,v_{j+1})=\gcd(v_1,v_{j+1}+k_{j+1}\tilde{v}_{j}),$$
     let $\tilde{v}_{j+1}=v_{j+1}+k_{j+1}\tilde{v}_{j}$, then $\gcd(v_1,\cdots, v_{j+1})=\gcd(v_1,\tilde{v}_{j+1})$.

    Let $ T=(t_{ij})\in GL_n(\Z) $, where $t_{ii}=1$ for  $i=1,\cdots,n$ and  $t_{ij}=\prod_{l=i+1}^{j}k_l$ for $2\leq i<j\leq n$, other $t_{ij}$s are $0$. That is

    $$T=\begin{bmatrix}
  1 & 0 &  0 & 0 & \cdots  & 0 \\
  0 & 1 & k_3 & k_3k_4 & \cdots  & \prod_{l=3}^{n}k_l \\
 0 & 0 & 1 & k_4 & \cdots  & \prod_{l=4}^{n}k_l  \\
  \vdots & \vdots & \vdots & \ddots & \ddots &  \vdots \\
0 & 0 & 0 & \cdots & 1 & k_n \\
  0 & 0 & 0 & \cdots & 0  & 1
\end{bmatrix}.$$
Note that
$\tilde{v}_{j}=v_{j} + k_{j}v_{j-1} + \cdots + (\prod_{l=i+1}^{j}k_l)v_i+\cdots+ (\prod_{l=3}^{j}k_l)v_2.$
One can verify $[v_1,v_2,\cdots,v_n]T=[v_1,\tilde{v}_2,\cdots,\tilde{v}_n]$ and $T^{-1}e_2=e_{2}$, where $e_{2}=[0,1,0,\cdots,0]^{T}$. Since  $A$ is of type $\mathcal{H}_0$, $v_1=a_{21}e_{2}$, then $T^{-1}v_1=v_1$. Hence
    $$ T^{-1}AT= T^{-1}[v_1,v_2,\cdots,v_n]T=T^{-1}[v_1,\tilde{v}_2,\cdots,\tilde{v}_n]=[v_1,T^{-1}\tilde{v}_2,\cdots,T^{-1}\tilde{v}_n], $$
and by Lemma \ref{lem:g}
    $$\gcd(v_1,T^{-1}\tilde{v}_n)=\gcd(T^{-1}v_1,T^{-1}\tilde{v}_n)=\gcd(v_1,\tilde{v}_n)=\gcd(v_1,\cdots,v_n)=\gcd(A)=1,$$
which means $T^{-1}AT$ is a matrix of type $\mathcal{H}_n$.
\end{proof}

\begin{lem}\label{lem:ma}
Let $A=(a_{ij})$ be an $n\times n$ integer matrix of type $\mathcal{H}_n$ satisfying one of the following conditions, then $m_A<n$.

(I) $a_{21}=0$, $a_{2n}=\cdots=a_{n-1,n}=0$, $a_{1n}\equiv \pm 1 \mod a_{nn}$;

(II) $a_{21}=0$, $a_{2n},  \cdots,a_{n-1,n}$ are not all 0;

(III) $a_{21}\neq 0$.
\end{lem}
\begin{proof}
We show below that there exists $v\in \mathbb{Z}^n$ such that $\{v,Av\}$ can be extended to a basis $\{v,Av, u_3,\cdots, u_n\}$ of $\mathbb{Z}^n$. Hence $\{v,u_3,\cdots, u_n\}$ is a generating orbit set of $A$ and $m_A<n$.

Let $v=[s,0,\cdots,0,t]^T$, since the shape of $A=(a_{ij})$ is
    $$ \begin{bmatrix}
        0& a_{12}& \cdots   & a_{1n} \\
        a_{21}& a_{22}& \cdots  & a_{2n} \\
        0 & a_{32} & \cdots & a_{3n} \\
        \vdots & \vdots & \vdots & \vdots \\
        0 & a_{n2} & \cdots & a_{nn}
    \end{bmatrix},$$
we have
$$Y=[v,Av]=\begin{bmatrix}
        s& 0& 0& \cdots   & 0 & t \\
        a_{1n}t& a_{21}s+ a_{2n}t& a_{3n}t & \cdots    &a_{n-1,n}t &a_{nn}t
    \end{bmatrix}^T.$$

The Smith normal form of $Y$ is
$$\begin{bmatrix}
        d_1(Y)& 0& 0&\cdots    & 0 \\
        0& d_2(Y)& 0&\cdots    & 0
    \end{bmatrix}^T$$
where $d_1(Y)=\gcd(Y)$ and  $d_2(Y)$ is the greatest common divisor of all $2\times 2$ minors of $Y$.
Note that $d_1(Y)$ divides $d_2(Y)$, we will choose suitable $s,t$ such that $d_2(Y)=1$, then the two columns $v,Av$ of $Y$ can be extended to a basis of $\Z^n$.

The $2\times 2$ minors of $Y$ that might be non-zero are as follows

    $$f_1(s,t)=a_{21}s^2+a_{2n}st=\begin{vmatrix}
        s & 0\\
        a_{1n}t & a_{21}s+ a_{2n}t
    \end{vmatrix},
    $$

  $$
    f'_1(s,t)=-(a_{21}st+a_{2n}t^2)=\begin{vmatrix}
        0 & t\\
        a_{21}s+ a_{2n}t & a_{nn}t
    \end{vmatrix},
    $$

    $$
    f_2(s,t)=a_{nn}st-a_{1n}t^2=\begin{vmatrix}
        s & t\\
        a_{1n}t & a_{nn}t\\
    \end{vmatrix},$$
and for $n>3$,
     $$
    f_j(s,t)=a_{jn}st=\begin{vmatrix}
        s & 0\\
        a_{1n}t & a_{jn}t
    \end{vmatrix},j=3,\cdots,n-1,
    $$
    $$
    f'_j(s,t)=-a_{jn}t^2=\begin{vmatrix}
        0 & t\\
        a_{jn}t & a_{nn}t
    \end{vmatrix},j=3,\cdots,n-1.
    $$

(I) If $a_{21}=a_{2n}=\cdots=a_{n-1,n}=0$ and $a_{1n}\equiv \pm 1 \mod a_{nn}$, then there exists $s_0\in \Z$ such that $a_{nn}s_0-a_{1n}=\pm 1$. Let $s=s_0,~t=1$, we have $\gcd(f_2(s,t))=1$, hence $d_2(Y)=1$.

(II) If $a_{21}=0$ and $a_{2n}, \cdots,a_{n-1,n}$ are not all 0, then $d=\gcd(a_{2n}, \cdots,a_{n-1,n})\neq 0$. By Lemma \ref{lem:v}, there exists $k\in \Z$ such that $$\gcd(d,k a_{nn}-a_{1n})=\gcd(d,a_{nn},-a_{1n})=\gcd(a_{21},a_{1n},\cdots,a_{nn})=1.$$ because $A$ is of type $\mathcal{H}_n$. Let $s=k,~t=1$, then for $n=3$
$$\gcd(f'_1(s,t),f_2(s,t))=\gcd(d,k a_{nn}-a_{1n})=1,$$
and for $n>3$
$$\gcd(f'_1(s,t),f_2(s,t),f'_3(s,t),\cdots, f'_{n-1}(s,t))=\gcd(d,k a_{nn}-a_{1n})=1.$$
We also have $d_2(Y)=1$.

(III) $a_{21}\neq 0$.
If $a_{1n}=a_{2n}=0$, let $s=t=1$, then $$\gcd(f_1(s,t),\cdots,f_{n-1}(s,t))=\gcd(a_{21},a_{1n},\cdots,a_{nn})=1$$ because $A$ is of type $\mathcal{H}_n$, and we have $d_2(Y)=1$.

Now $c_1=\gcd(a_{1n},a_{2n})\neq 0$, there exist $k,\ell\in\mathbb{Z}$ such that $ka_{1n}+\ell a_{2n}=c_1$. The matrix $T=\begin{bmatrix}
        \ell  & -k \\
        a_{1n}/c_1 & a_{2n}/c_1\\
    \end{bmatrix}$ is in $GL_2(\mathbb{Z})$ and $$T\begin{bmatrix}
        a_{21} & a_{2n}\\
        a_{nn} & -a_{1n}
    \end{bmatrix}=\begin{bmatrix}
        c_2 & c_1\\
        c_3 & 0\\
    \end{bmatrix}$$
where $c_2=\ell a_{21}-k a_{nn},~ c_3=(a_{1n}a_{21}+a_{2n}a_{nn})/c_1$. By Lemma \ref{lem:g},
    \begin{equation}\label{c123}
        \gcd(c_1,c_2,c_3)=\gcd(\begin{bmatrix}
        c_2 & c_1\\
        c_3 & 0\\
    \end{bmatrix})=\gcd(\begin{bmatrix}
        a_{21} & a_{2n}\\
        a_{nn} & -a_{1n}
    \end{bmatrix})=\gcd(a_{21},a_{1n},a_{2n},a_{nn}).
    \end{equation}

Note that $$\begin{bmatrix}
        a_{21} & a_{2n}\\
        a_{nn} & -a_{1n}
    \end{bmatrix}\begin{bmatrix}
        s \\
        t
    \end{bmatrix}=\begin{bmatrix}
        \ell  & -k\\
        a_{1n}/c_1 & a_{2n}/c_1\\
    \end{bmatrix}^{-1}\begin{bmatrix}
        c_2 & c_1\\
        c_3 & 0\\
    \end{bmatrix}\begin{bmatrix}
        s \\
        t
    \end{bmatrix},$$
    i.e.
    $$\begin{bmatrix}
        a_{21}s+ a_{2n}t\\
        a_{nn}s  -a_{1n}t
    \end{bmatrix}=\begin{bmatrix}
        \ell  & -k\\
        a_{1n}/c_1 & a_{2n}/c_1\\
    \end{bmatrix}^{-1}\begin{bmatrix}
         t c_1+s c_2\\
        sc_3
    \end{bmatrix}
    ,$$
by Lemma \ref{lem:g},
           \begin{equation}\label{eq0}
             \gcd(a_{21}s+ a_{2n}t,a_{nn}s-a_{1n}t) = \gcd( tc_1+sc_2,sc_3).
           \end{equation}

Moreover, for any two coprime integers $s,t$ such that $\gcd(a_{21},t)=1$,  we have $\gcd(a_{21}s+a_{2n}t, t)=1$ and so
           \begin{equation}\label{eq1}
             \gcd(a_{21}s+a_{2n}t, t(a_{nn}s  -a_{1n}t)) = \gcd(a_{21}s+a_{2n}t, a_{nn}s-a_{1n}t).
           \end{equation}
Since $$\gcd(f_1(s,t),f'_1(s,t))=(a_{21}s+ a_{2n}t)\gcd(s,-t)=a_{21}s+ a_{2n}t,$$
by the equalities (\ref{eq1}) and (\ref{eq0}), we have
\begin{align}\label{eq2}
  \gcd(f_1(s,t),f'_1(s,t),f_2(s,t)) &= \gcd(a_{21}s+ a_{2n}t, t(a_{nn}s-a_{1n}t)) \notag \\
   &= \gcd(a_{21}s+a_{2n}t, a_{nn}s-a_{1n}t) \notag\\
   &= \gcd( tc_1+sc_2,sc_3).
\end{align}


%

     Let $c_4=\begin{cases}
                \gcd(a_{3n},\cdots,a_{n-1,n}), & \mbox{if } n>3 \\
                0, & \mbox{if } n=3
              \end{cases}$ and $\mathcal{P}=\{p \text{ is prime},~p\mid \gcd(c_3,c_4)\}$. There are two cases:


    Case 1:  $\mathcal{P}$ is infinite which means $c_3=c_4=0$, then $\gcd(a_{21},a_{1n},a_{2n},a_{nn})=1$ since $A$ is of type $\mathcal{H}_n$. By the equality (\ref{c123}),  $\gcd(c_1,c_2)=\gcd(c_1,c_2,c_3)=1$ and there exist $s_0,t_0\in \Z$ with $\gcd(s_0,t_0)=1$ such that $t_0 c_1+s_0 c_2=1$.

    Note that $\gcd(t_0,c_2)=1$ and $a_{21}\neq 0$, by Lemma \ref{lem:v}, there exists $k\in \Z$ such that $1=\gcd(a_{21}, c_2,t_0)=\gcd(a_{21}, t_0+kc_2)$. Let $s=s_0-kc_1, t=t_0+kc_2$, then $\gcd(a_{21},t)=1$ and $tc_1+sc_2=1$. We have $\gcd(s,t)=1$. By the equality (\ref{eq2}), $\gcd(f_1(s,t),f'_1(s,t),f_2(s,t))=1$, so $d_2(Y)=1$.

%

   Case 2: $\mathcal{P}$ is finite. $\mathcal{P}$ can be divided into three disjoint subsets as follows
    \begin{align*}
        &\mathcal{P}_1=\{p\in\mathcal{P},~p\nmid c_1\},\\
        &\mathcal{P}_2=\{p\in\mathcal{P},~p\mid c_1,~p\nmid a_{21}\},\\
        &\mathcal{P}_3=\{p\in\mathcal{P},~p\mid c_1,~p\mid a_{21}\}.
    \end{align*}

    Let $s=\begin{cases}
             1, & \mbox{if } \mathcal{P}_1=\varnothing, \\
             \prod_{p\in\mathcal{P}_1}p, & \mbox{otherwise},
           \end{cases}$
    and  $t=\begin{cases}
             1, & \mbox{if } \mathcal{P}_2=\varnothing, \\
             \prod_{p\in\mathcal{P}_2}p, & \mbox{otherwise}
           \end{cases}$, then $\gcd(a_{21},t)=1$ and $\gcd(s,t)=1$.

   Suppose $p$ is a  prime number that divides $\gcd(s c_3,tc_4)$, then $p\in \mathcal{P}$.
\begin{itemize}
  \item If  $p\in \mathcal{P}_1$, then $p\mid s$  and $p\nmid t$, hence $p\nmid (tc_1+sc_2)$.
  \item If $p\in \mathcal{P}_2\cup \mathcal{P}_3$, then $p\nmid s$ and $p\mid \gcd(c_1,c_3,c_4)$. Since $\gcd(c_1,c_2,c_3,c_4)=1$, $p\nmid c_2$, hence $p\nmid tc_1+sc_2$.
\end{itemize}
    The above argument shows that $\gcd(sc_3,tc_4,tc_1+sc_2)=1$.

    If $n=3$, then $c_4=0$ and by the equality (\ref{eq2}), we have
    $$\gcd(f_1(s,t),f'_1(s,t),f_2(s,t))=\gcd(sc_3,tc_1+sc_2)=\gcd(sc_3,tc_4,tc_1+sc_2)=1.$$

    If $n>3$,
note that $$\gcd(f_3(s,t),f'_3(s,t),\cdots,f_{n-1}(s,t),f'_{n-1}(s,t))=tc_4,$$ by the equality (\ref{eq2}) we have
$$\gcd(f_1(s,t),f'_1(s,t),f_2(s,t),f_3(s,t),f'_3(s,t),\cdots,f_{n-1}(s,t),f'_{n-1}(s,t))=1.$$
Hence $d_2(Y)=1$.
%
\end{proof}

\begin{cor}\label{cor:hn}
Suppose $H=(h_{ij})$ is an element in $GL_n(\Z)$ such that $H$ is of type $\mathcal{H}$  with $\gcd(H-h_{11}I)=1$, then $m_H<n$.
\end{cor}
\begin{proof}
  $H-h_{11}I$ is of type $\mathcal{H}_0$ with $\gcd(H-h_{11}I)=1$, by Proposition \ref{lem:hn}, there exists $T\in GL_n(\Z)$, such that $$T^{-1}(H-h_{11}I)T=T^{-1}HT-h_{11}I$$ is a matrix of type $\mathcal{H}_n$. Let $H'=T^{-1}HT=(h'_{ij})$ and $H''=H'-h_{11}I=(h''_{ij})$, by Lemma \ref{cor:l} and Lemma \ref{lem:co}, we have $m_{H}=m_{H'}=m_{H''}$.

  If $h''_{21}, h''_{2n}, \cdots, h''_{n-1,n}$ are not all $0$, then $H''$ satisfies the condition (II) or (III) of Lemma \ref{lem:ma}, hence $m_{H}=m_{H''}<n$.

  Otherwise, $h''_{21}=h''_{2n}=\cdots=h''_{n-1,n}=0$, we have $\gcd(h''_{1n},h''_{nn})=1$ because $H''$ is of type $\mathcal{H}_n$. Moreover, $h'_{21}=h'_{2n}=\cdots=h'_{n-1,n}=0$ and the shape of $H'$ is
    $$ \begin{bmatrix}
        h'_{11} & \ast & h'_{1n} \\
        0 & H^{\ast} & 0 \\
        0 & \ast & h'_{nn}
      \end{bmatrix}.$$
  Since $\pm 1=\det(H)=\det(T^{-1}HT)=\det(H')=h'_{11}h'_{nn}\det(H^{\ast}),$
   we have $h'_{11}=\pm 1$ and $h'_{nn}=\pm 1$.

   Note that $h''_{11}=h'_{11}-h_{11}=0$, then $h''_{nn}=h'_{nn}-h_{11}=h'_{nn}-h'_{11}=0$ or $\pm 2$. If $h''_{nn}=0$, then $|h''_{1n}|=\gcd(h''_{1n},h''_{nn})=1$,  so $h''_{1n}\equiv \pm 1 \mod h''_{nn}$. If $h''_{nn}=\pm 2$, then $h''_{1n}$ is odd because $\gcd(h''_{1n},h''_{nn})=1$, we also have $h''_{1n}\equiv \pm 1 \mod h''_{nn}$. Thus $H''$ satisfies the condition (I) of Lemma \ref{lem:ma} and $m_{H}=m_{H''}<n$.
\end{proof}

\begin{thm}\label{mthm2}
Suppose $A=(a_{ij})$ is an element in $GL_n(\Z)$ with $n\geq 3$, then $m_A=n$ if and only if $\gcd(A-a_{11}I)\neq 1$.
\end{thm}
%
\begin{proof}
    If $\gcd(A-a_{11}I)=1$, by Proposition \ref{prop:h}, there exists $T\in GL_n(\Z)$ such that $H=T^{-1}AT=(h_{ij})$ is of type $\mathcal{H}$ with $h_{11}=a_{11}$. By Lemma \ref{lem:g}, $\gcd(H-h_{11}I)=\gcd(T^{-1}AT-a_{11}I)=\gcd(T^{-1}(A-a_{11}I)T)=\gcd(A-a_{11}I)=1$, thus by Lemma \ref{lem:co} and Corollary \ref{cor:hn}, $m_A=m_H<n$.

 If $\gcd(A-a_{11}I)\neq 1$, then by Lemma \ref{cor:l} and \ref{lem:n}, $m_A=m_{A-a_{11}I}=n$.
\end{proof}

\noindent\textbf{Acknowledgements.} The authors are partially supported by National Natural Science Foundation of China (No. 12271385).

\end{document}